\documentclass{siamltex}
\usepackage{graphicx}
\usepackage{epsfig}
\usepackage{float}
\usepackage[ruled,vlined]{algorithm2e}
\usepackage{hyperref}
\usepackage{appendix}
\usepackage{comment}

\usepackage{psfrag} 

\usepackage{amsmath,amsfonts}
 \usepackage{capt-of}
 \usepackage{color,dsfont,bbm}
\usepackage[normalem]{ulem}
\usepackage{cancel}

\usepackage{booktabs}
\usepackage{xcolor}
\usepackage{comment}
\usepackage{xcolor}
\usepackage{comment}

\def\real{{\mathbf R}}

\def\d{{\mathrm d}}
\def\e{{\mathrm e}}
\def\iu{\mathrm{i}}
\def\eps{\varepsilon}

\newdimen\GGGlength
\newdimen\GGGheight
\newbox\GGGbox
\def\GGGput[#1,#2](#3,#4)#5{%
  \setbox\GGGbox\vbox{\hbox{#5}\kern0pt}%
  \GGGlength\wd\GGGbox%
  \divide\GGGlength by100 \multiply\GGGlength by#1%
  \GGGheight\ht\GGGbox%
  \divide\GGGheight by100 \multiply\GGGheight by#2%
  \put(#3,#4){\kern-\GGGlength\raise-\GGGheight\box\GGGbox}}

\newcommand{\R}{\mathbb R}

\newcommand{\bq}{\begin{equation}}
\newcommand{\eq}{\end{equation}}

\newcommand{\tol}{\mbox{\it Tol\/}}

\newtheorem{thm}{Theorem}
\newtheorem{remark}{Remark}

\def\EH{\color{black}}
\def\NG{\color{black}}
\def\NG{\color{black}}

\def\iu{{\rm i}}
\def\e{{\rm e}}

\def\d{{\rm d}}
\def\bigo{{\mathcal O}}

\def\real{\mathbb R\hspace{1pt}}

\def\1{{\bf 1}}

\title{A fast and memoryless numerical method for solving fractional differential equations}

\author{Nicola Guglielmi\footnotemark[1] \and Ernst Hairer\footnotemark[2]}

\begin{document}

\maketitle

\renewcommand{\thefootnote}{\fnsymbol{footnote}}
\footnotetext[1]{Division of Mathematics, 
Gran Sasso Science Institute,
Via Crispi 7,
I-67100   L'Aquila,  Italy. Email: {\tt nicola.guglielmi@gssi.it}}
\footnotetext[2]{Section de Math{\'e}matiques, Universit\'e de Gen{\`e}ve,
       CH-1211 Gen{\`e}ve 24, Switzerland.\\
       Email: {\tt ernst.hairer@unige.ch}}
\renewcommand{\thefootnote}{\arabic{footnote}}

%
%
%
%
%
%
%
%
%
%
%
%
\begin{abstract}
The numerical solution of implicit and stiff differential equations by implicit numerical
integrators has been largely investigated and there exist many excellent efficient
codes available in the scientific community, as {\sc Radau5} (based on a Runge-Kutta
collocation method at Radau points) and {\sc Dassl}, based on backward differentiation formulas,
among the others.
When solving fractional ordinary differential equations (ODEs),
the derivative operator is replaced by
a non-local one and the fractional ODE is reformulated as a Volterra integral equation,
to which these codes cannot be directly applied.

This article is a follow-up of the article by the authors \cite{guglielmi25asi} for differential
equations with distributed delays.
The main idea is to approximate the fractional kernel $t^{\alpha -1}/
\Gamma (\alpha )$
($\alpha >0$) 
by a sum of exponential functions or by a sum of exponential functions multiplied by
a monomial, and then to transform the fractional integral (of convolution type)
into a set of ordinary differential equations. 
The augmented system is typically stiff and thus
requires the use of an implicit method.
It can have a very large dimension and requires a special treatment of the
arising linear systems.

The present work presents an algorithm for the construction of
an approximation of the fractional kernel by a sum of exponential functions,
and it shows how the arising linear systems in a stiff time integrator
can be solved efficiently.
It is explained how the code {\sc Radau5} can be used
for solving fractional differential
equations.
Numerical experiments illustrate the accuracy and the
efficiency of the proposed method.
Driver examples
are publicly available from the homepages of the authors.

\end{abstract}



\pagestyle{myheadings} \thispagestyle{plain}
\markboth{N. GUGLIELMI,  E. HAIRER}{Solving fractional differential equations by implicit ODE solvers}


\section{Introduction}
\label{sect:intro}

The study of fractional integrals and fractional differential equations (FDEs) has developed
incredibly in recent years; there is now a very large number of articles dealing
with various versions of fractional derivatives and models, as well as their analysis
and application. In parallel, the study of suitable numerical integrators is also a
very open and active area of research.

In this paper we will discuss Initial Value Problems (IVPs) mainly for the Caputo fractional
derivative, but also for the Riemann-Liouville  fractional derivative, the two fractional
derivatives that are most commonly used, both are defined in terms of the
Riemann-Liouville fractional integral. 
Many results can be found in textbooks such as  \cite{Pod99,diethelm-bk,kst,skm,Bang21}.

Fractional differential equations are usually written as a special case of
Volterra integral and integro-differential equations. There is a large literature on the numerical
discretization of such problems, and their accuracy and stability is well investigated (see the monographs by H.\ Brunner
\cite{brunner04cmf,brunner86tns}).
Especially for problems with weakly singular kernel the technique of
discretized fractional calculus has been
developped by C.\ Lubich \cite{lubich86dfc} (see also \cite{hairer87nmf}),
and  the oblivious convolution quadrature \cite{Sc06,LF08} for more general kernels.

{\EH Further methods which are proposed and analyzed in the literature (see e.g. \cite{Huang2018}), are called $L1$ methods. They
are some of the simplest and most widely used schemes for time-fractional problems. 
The $L1$ method is a direct quadrature-based approach for approximating the Caputo fractional derivative. It uses piecewise linear interpolation of the solution
and approximates the Caputo derivative on a uniform or graded time mesh.
The fast $L1$ method uses an approximation of the fractional kernel by a sum of
exponential functions.
}

Available codes mostly make use of constant step size.
For a survey about numerical methods for fractional differential equations, we refer the reader e.g.\ to
\cite{Indam} and the references therein.
A typical drawback of available methods is the complexity in dealing with the memory associated to
the fractional integral, which increases as time increases and the inherent difficulty to make use of
variable stepsize strategies. 
However, in realistic situations one is often confronted
with initial layers and/or fast transitions between different states, so that
flexibility in the choice of step size is an essential ingredient
for efficiency.
Although there exist excellent codes for solving nonstiff and stiff ordinary differential equations,
codes for differential-algebraic equations,  we are not aware of a code based on a variable step size
time integrator that can efficiently treat problems with fractional derivatives.

Our approach consists in modifying the fractional differential equation in such a way
that any code for stiff and differential-algebraic equations
can be applied. Following the ideas of \cite{guglielmi25asi}
the fractional kernel in the integral term is approximated by
a sum of exponential functions.
The integral term is then transformed into a set of differential equations.
In this way, fractional differential equations can be solved efficiently and reliably by standard software 
for stiff and differential-algebraic equations. The method we propose is an oblivious method which allows to replace the memory by a system of additional ODEs, solving in this way all the technical difficulties related to handling the memory term and to the use of variable step sizes.
This is illustrated at the hand of {\sc Radau5} \cite{hairer96sod} for stiff and differential-algebraic 
equations.

\subsubsection*{Related literature}

There are essentially two approaches for approximating the
fractional kernel $k(t) = t^{\alpha-1}/\Gamma (\alpha )$ by a
sum of exponential functions. Since the Laplace transform of $k(t)$ is
$\lambda^{-\alpha}$ for $0 < \alpha < 1$, the inversion formula gives
\begin{equation}\label{invlap}
 k(t) = \frac{t^{\alpha -1}}{\Gamma (\alpha )} = 
 \frac 1 {2\pi \iu}\int_{\Lambda}
 \e^{t\lambda} \lambda^{-\alpha} \, \d \lambda \approx \sum_{j=1}^n
 w_j \,\e^{t\lambda j},
\end{equation}
where $\Lambda$ is a suitable path in the complex plane. Discretizing the
integral gives the desired sum of exponential functions, where $w_j$ and $\lambda_j$
are complex numbers.
The early publication \cite[Formula (1.6)]{lubich88cq1}
(see also \cite{lubich93rkm}) proposes to
insert this formula into the integral term
$\int_0^t k(t-s) f\bigl( s, y(s)\bigr)\,\d s$ (see Section~\ref{subsect:fde} below),
which then can be transformed into a series
of scalar linear ordinary differential equations.
Based on such an approximation, a variable step size algorithm
for the numerical solution
of fractional differential equations is presented
in \cite{LF08}, for which advancing $N$ steps requires only
$\bigo(N \log N)$ computations and $\bigo(\log N )$ active memory.

In \cite{BaHe17,BaHe17b}, the authors split the fractional integral into a local term
(containing the singularity), $\int_0^\delta k(t-s) f\bigl( s, y(s)\bigr)\,\d s$ for a small $\delta > 0$, and a more expensive history term $\int_\delta^t k(t-s) f\bigl( s, y(s)\bigr)\,\d s$. 
For the history term they consider a multipole
approximation of the Laplace transform of the kernel, which then leads to
an approximation of the form \eqref{invlap} with complex $w_j$ and
$\lambda_j$. Again a transformation to a series of scalar linear
differential equations makes the algorithm efficient.

Another approach is to use the fact that the Laplace transform
of the function
$z^{-\alpha }$ is $\Gamma (1-\alpha )\, t^{\alpha -1 }$
(for $0 < \alpha < 1$), which gives
the real integral re\-pre\-sen\-ta\-ti\-on
\footnote{Note that the integral representation formula \eqref{lapintro}  can also be derived from \eqref{invlap} as a real inversion formula, as is discussed in \cite[Chapter 10, page 284]{Henrici2}, specifically in Theorem 10.7d and in Example 4, for $0 < \alpha < 1$.}
\begin{equation}\label{lapintro}
k(t) = \frac {t^{\alpha -1}}{\Gamma (\alpha )} = \frac{\sin (\pi\alpha)}{\pi} \int_0^\infty \e^{-tz}z^{-\alpha } \,\d z \approx\sum_{j=1}^n c_j\,
\e^{-\gamma_j t},
\end{equation}
with real numbers $c_j$ and $\gamma_j$.
To get the approximation by a sum of exponentials it is proposed in
\cite{Li10} to first remove the singularity with the change of
variables $\zeta = z^{1-\alpha}$ (for $0<\alpha < 1$), then to write
$(0,\infty )$ as a distinct union of geometrically
increasing intervals, and finally to apply to each of these intervals
(away from $0$) a Gauss--Legendre quadrature.

In \cite{Baf18}, a composite Gauss--Jacobi is applied directly to
the integral in \eqref{lapintro} on geometrically increasing intervals.
When applied to fractional differential equations,
an integral deferred correction scheme is combined with an L-stable
diagonally implicit Runge--Kutta scheme for the approximation
of the differential equations obtained by the kernel compression.
In \cite{BL19},  the authors propose a fast and memory efficient numerical method for solving a fractional integral based on a Runge-Kutta convolution quadrature. The method is applied to the numerical solution of
fractional diffusion equations in space dimension up to $3$.

{\NG
A further class of fast methods is analyzed e.g. in \cite{Jiang2017,Jiang2018}, where - based on exponential expansions - a first and a second order efficient schemes are proposed, with a constant stepsize, i.e. on a uniform grid.  
They are applied for solving time-fractional diffusion PDEs.
}

The approach that we follow in the present work is that of \cite{beylkin10abe}.
It uses the real integral representation \eqref{lapintro}, performs the
change of variables $z=\e^s$ to get an integral over the
real line $(-\infty , \infty)$,
and finally applies the trapezoidal rule. This gives simple expressions
for the coefficients $c_j$ and $\gamma_j$ in the sum of exponentials.

{\EH
Fractional partial differential equations are important in applications.
Wave propagation in a viscoelestic medium
with singular memory is studied in \cite{lu04nmm} and a numerical approach is presented.
As long as the dimension in space is one, we are lead
after space discretization to problems with banded
Jacobians and the code of the present work can be applied effectively. 
For higher space dimensions, an adaptation of the linear algebra routines is recommended for reasons of efficiency.
Fractional derivatives are also used in treating transparent boundary conditions
\cite{antoine08aro}.
}

\subsubsection*{Outline of the paper}

Section \ref{sect:dgln-eh} recalls the definitions of Caputo
and Riemann-Liouville fractional derivatives and of a
class of Cauchy problems for fractional differential equations.
These problems are written in the form of integro-differential
equations that are special cases of a general formulation which
can be treated with the approach of the present work.
Section~\ref{sect:fdetoode} explains the linear chain trick, which
permits to transform the integro-differential equations (with fractional
kernel in the convolution integrals) into a large system of stiff
ordinary differential equations. The application of the code
{\sc Radau5} for solving fractional differential equations is
explained in Section~\ref{sect:solvradau} with emphasis on the
efficient solution of the arising linear systems.
Section~\ref{sect:approxsumexp} presents explicit formulas
by Beylkin and Monz\' on \cite{beylkin10abe} for
an approximation of the fractional kernel $k(t) =t^{\alpha -1}/\Gamma(\alpha)$
$(0<\alpha < 1)$
by a sum of exponential functions, and it discusses the effect
of such an approximation on the solution of the problem.
The case $\alpha > 1$ is considered in Section~\ref{sect:alphage1}.
Two approaches, one by splitting the kernel and the other by
differentiating the integral are discussed.
The final Section~\ref{sect:numerical} presents numerical experiments
for a scalar test equation with known exact solution, for a
modified Brusselator reaction, for a linear fractional
partial differential equation, and for a system of reaction-diffusion
problems.

\section{Fractional differential equations}
\label{sect:dgln-eh}

This section is devoted to a general formulation of fractional differential equation,
which can be solved numerically by the techniques of the present work.

\subsection{Integral and differential fractional operators}
For a given real number $\alpha > 0$ and for $t>0$ the {\it Riemann-Liouville
fractional integral} of order
$\alpha$ is defined as an integral (with a possibly weakly singular kernel, if $\alpha < 1$) by
\begin{equation}\label{eq:RLint}
  J^{\alpha} f(t)=\frac{1}{\Gamma(\alpha)}\int_0^t (t-s)^{\alpha-1} f(s)\,ds .
\end{equation}
For $\alpha = 1$ we recover the usual integral, and for integer $\alpha$ we have
an iterated integral.

For a non-integer $\alpha > 0$ we let $m=\lceil \alpha \rceil$ be the smallest integer that is
larger than~$\alpha$.
If $D$ denotes the derivative with respect to $t$, the {\it Riemann-Liouville
fractional derivative}
is then given by $D^\alpha = D^m J^{m-\alpha}$, i.e.,
\begin{equation} \label{eq:RLder}
 D^{\alpha} f(t) =  \frac{d^m}{d t^m} \biggl(
\frac{1}{\Gamma(m-\alpha)} \int_0^t (t-s)^{m-\alpha-1} f(s)\,ds \biggr) .
\end{equation}

When considering fractional differential equations it is convenient to use a slightly
different definition of fractional derivative. The {\it Caputo fractional
derivative} is defined by $D_*^\alpha = J^{m-\alpha} D^m$, i.e.,
\begin{equation} \label{eq:Cder}
D_*^{\alpha}f(t) =  
\frac{1}{\Gamma(m-\alpha)}\int_0^t (t-s)^{m-\alpha-1} f^{(m)}(s)\,ds.
\end{equation}
The relation between the Riemann-Liouville and Caputo fractional derivatives is
\begin{equation}\label{connect}
D_*^\alpha f(t) = D^\alpha \Bigl( f(t) - T_{m-1} (t)\Bigr) , 
\qquad T_{m-1} (t) = \sum_{k=0}^{m-1} \frac{t^k}{k!} \, f^{(k)} (0) .
\end{equation}
More information about fractional differential and integral operators can be found
in the monographs \cite{diethelm-bk} and \cite{gorenflo97fci}.

\subsection{Fractional differential equations}\label{subsect:fde}
We begin by considering initial value problems of the form
\begin{equation}\label{eq:fde}
  D_*^{\alpha}y(t) = f\bigl(t,y(t)\bigr), \qquad
   y^{(k)}(0) = y_{k},\quad k=0,\ldots ,m-1 ,
\end{equation}
where $f$ is sufficiently regular, $\alpha >0$ and $m=\lceil \alpha \rceil$.
For the existence and uniqueness of a solution we refer to \cite{diethelm-bk}.
For a numerical treatment it is convenient to write such a problem as a Volterra
integral or Volterra integro-differential equation.
There are several ways to do this.

We can apply the operator $J^\alpha$ to \eqref{eq:fde} and use
$J^\alpha D_*^\alpha f(t)= J^m D^m f(t) = f(t) - \sum_{k=0}^{m-1} f^{(k)} (0) t^k/k!$.
This yields the Volterra integral equation
\begin{equation} \label{Volt1}
y(t) = T_{m-1}(t) + \frac{1}{\Gamma(\alpha)}\int_0^{t} (t-s)^{\alpha-1}f\big(s,y(s)\big)\,ds.
\end{equation}
Differentiating this relation $(m-1)$ times with respect to $t$, gives an
equivalent Volterra integro-differential equation (if $\alpha$ is not an 
integer and $m\ge 1$)
\begin{equation} \label{Volt2}
y^{(m-1)}(t) =  y_{m-1} + \frac{1}{\Gamma(\alpha - m  +1)}\int_0^{t} 
(t-s)^{\alpha-m}f\bigl( s,y(s)\bigr)\,ds,
\end{equation}
with initial values $y^{(k)} (0) = y_k$ for $k=0, \ldots , m-2$.

For the case that the fractional derivative is not the highest
derivative in the equation
(see the example of Section~\ref{sect:multi}), it is possible to insert
directly the definition~\eqref{eq:Cder} of the frcational derivative
to get a Volterra integro-differential equation.

\subsection{General formulation as integro-differential equations}
The aim of this work is to present a numerical approach that permits us to
accurately solve initial value problems of the form
\begin{equation}\label{eq:Voltgeneral}
    M \dot y (t) = F\bigl(t, y(t), I(y)(t)\bigr) , \qquad y(0) = y_0 .
\end{equation}
Here, $y(t) $ is a vector in $\real^{d}$,
$M$ is a possibly singular square matrix,
$F(t,y,I)$ is sufficiently regular, and $I(y)(t)$ is a vector
in $\real^{d_I}$ with components
\begin{equation}\label{integral}
I_j (y) (t) = \frac 1 {\Gamma (\alpha_j)}\int_0^t (t-s)^{\alpha_j -1} 
G_j \bigl( s, y(s)\bigr) \, \d s ,
\end{equation}
where $\alpha _j > 0$, and $G_j(t,y)$ is smooth as a function of $y$.

Both integral formulations of Section~\ref{subsect:fde} are special
cases of \eqref{eq:Voltgeneral}, and each of them
has $d_I = d$.
For the formulation \eqref{Volt1} we put $M=0$, $\alpha_j = \alpha$, and
$G_j(t,y) = f_j(t,y)$, where $f_j(t,y)$ denotes the $j$th component
of $f(t,y)$.

For the formulation \eqref{Volt2} with $m\ge 2$ we introduce the variables $u_j(t)= y^{(j)}(t)$
for $j=1,\ldots , m-2$, so that $\dot u_{m-2} (t)$ becomes the right-hand side
of \eqref{Volt2}. To close the system we have to add
the differential equations $\dot y (t) = u_1 (t)$
and $\dot u_j(t) = u_{j+1}(t)$ for $j=1,\ldots , m-3$.
We get a system \eqref{eq:Voltgeneral}
with solution vector $\bigl( y(t), u_1 (t), \ldots , u_{m-2}(t)\bigr)$.
The matrix $M$ is the identity, and the integral terms have $\alpha_j =
\alpha - m +1 \in (0,1)$ and $G_j(t,y) = f_j(t,y)$ as before.

The system \eqref{eq:Voltgeneral} comprises not only all formulations of
Section~\ref{subsect:fde} as special cases, but it is much more general.
The system can be a differential-algebraic one, and several
fractional derivatives with different values of $\alpha_j$ are permitted.

\section{Ordinary differential equations approximating fractional problems}
\label{sect:fdetoode}

Instead of applying a quadrature formula to the integral in \eqref{integral},
our approach consists of approximating the kernel $k(t) = t^{\alpha -1}/\Gamma(\alpha)$ for $0 < \alpha < 1$, by a sum
of exponential functions, so that the integro-differential equation \eqref{eq:Voltgeneral}
can be reduced to a system of ordinary differential equations.

\subsection{Linear chain trick}
\label{sect:chain}

Assume that we have at our disposal an approximation
\begin{equation}\label{eq:sum-exp}
\frac {t^{\alpha -1}}{\Gamma (\alpha )} \approx \sum_{i=1}^n c_i \,\e^{-\gamma_i t}
\end{equation}
with real coefficients $c_{i}$ and $\gamma_i > 0$ depending on $\alpha$ and on the
accuracy of the approximation.
Explicit formulas for $c_i$ and
$\gamma_i$ as well as a study of the effect of such an approximation to the solution
of the fractional differential equation are given
in Section~\ref{sect:approxsumexp}. Replacing in \eqref{integral} the kernel
by \eqref{eq:sum-exp} yields the approximation
\begin{equation}\label{eq:intassumexp}
\widetilde I_j(y)(t) = \sum_{i=1}^{n_j} c_{i,j} \, z_{i,j} (t), \qquad z_{i,j} (t) =
\int_0^t \e^{-\gamma_{i,j} (t-s)}\, G_j\bigl( s, y(s)\bigr) \,\d s ,
\end{equation}
where for $\alpha = \alpha_j$ the parameters in \eqref{eq:sum-exp} are $n_j$,
$c_{i,j}$ and $\gamma_{i,j}$.
Depending on $G_j(t,y)$, the functions $z_{i,j}(t)$ are scalar or vector-valued.
Differentiation of $z_{i,j}(t) $ with respect to time yields, for $i=1,\ldots ,n_j$
and $j=1,\ldots , d_I$,
\begin{equation}\label{zi-dgl}
\dot z_{i,j} (t) = - \,\gamma_{i,j}\, z_{i,j}(t) + G_j\bigl( t, y(t)\bigr)   .
\end{equation}
We now substitute $I_j(t) = \widetilde I_j(y)(t)$ from \eqref{eq:intassumexp}
for $I_j(y)(t)$ into the equation \eqref{eq:Voltgeneral} to obtain
\begin{equation}\label{eq:ODE}
\begin{array}{rcl}
    M \dot y(t) & = & F\bigl( t,y(t),I_1(t),\ldots , I_{d_I}(t)\bigr) 
    \qquad
     I_j(t) =\sum_{i=1}^{n_j} c_{i,j} \, z_{i,j} (t)\\[2mm]
\dot z_{i,j} (t) &=& - \,\gamma_{i,j}\, z_{i,j}(t) + G_j\bigl( t, y(t)\bigr),
\quad i=1,\ldots ,n_j ,\quad j=1,\ldots , d_I .\end{array} 
\end{equation}
Initial values are $y(0)=y_0$ from \eqref{eq:Voltgeneral}, and
$z_{i,j} (0)=0$ for all $i$ and $j$.

This is a system of ordinary differential equations
\begin{equation}\label{eq:ODEshort}
    {\cal M} \,\dot Y(t)\, = \,{\cal F}\bigl( t, Y(t)\bigr),\qquad Y(0)=Y_0
\end{equation}
for the vector $Y(t) = \bigl( y(t),z_{i,j}(t)\bigr) $, where $j=1,\ldots ,d_I$
and $i=1,\ldots ,n_j$. Here, ${\cal M}$ is the diagonal matrix with entries from $M$
for the $y$-variables,
and with entries $1$ for the $z$-variables. The system is of dimension
$d+n_1+\ldots + n_{d_I}$.
Since some coefficients among the $\gamma_{i,j}$ are typically very large and positive,
this system is stiff
independent of whether the original equation is stiff or not.

\subsection{Structure of the Jacobian matrix}
\label{sect:structjac} 

For notational convenience we assume here that all functions $G_j(t,y) $ are scalar.
This can be done without loss of generality by eventually increasing the number
of integrals in \eqref{eq:Voltgeneral}.
The Jacobian matrix of the vector field in \eqref{eq:ODEshort} is
\begin{equation}\label{eq:jacODE}
{\cal J} = \frac {\partial\cal F}{\partial Y} =
\begin{pmatrix}
    J &  \ C_1 \ & \ C_2 \ &\cdots & \ 
    C_{d_I}  \ \\[1mm]
    B_1  & J_1 & 0 &\cdots & 0 \\[-1mm]
    B_2 & 0 & J_2  &\ddots & \vdots \\[-1mm]
    \vdots &  \vdots & \ddots & \ddots &0 \\[1mm]
    B_{d_I} & 0 & \cdots & 0& J_{d_I}
\end{pmatrix} ,
\end{equation}
where $J=\partial F/\partial y $ is the $d$-dimensional derivative matrix,
$C_j = (\partial F/\partial I_j) \,c_j^\top$ a rank-one matrix
with $c_j^\top = (c_{1,j}, \ldots , c_{n_j,j}$),
$B_j = e \,(\partial G_j/\partial y) $ a rank-one matrix
($e$ is the $n_j$-dimensional
column vector with all entries equal to $1$), and $J_j$ is the $n_j$- dimensional
diagonal matrix
with entries $-\gamma_{1,j}, \ldots , -\gamma_{n_j,j}$.

\begin{remark}\rm \label{rem:intd}
If one is interested in the values of the integral $I_j(t)$ along
the solution, one can introduce $y_{d+1} (t) =I_j(t)$ as a new dependent
variable, replace $I_j$ by
$y_{d+1} $ in \eqref{eq:ODE}, and add the algebraic relation
$0= \sum_{i=1}^{n_j} c_{i,j} \, z_{i,j} (t) - y_{d+1}(t)$
to the system. In this way the dimension of the vectors $y$ and
$F$ is increased by $1$, and an additional diagonal element $0$
is added to the matrix $M$. This procedure can be applied to several or
to all integrals $I_j$. 

This reformulation of the
problem has been used in \cite{guglielmi25asi} to monitor the accuracy
of the integrals in the system. If the number of considered integrals is large,
it can increase the overhead of the computation.
\end{remark}

\section{Solving fractional differential equations by the code RADAU5}
\label{sect:solvradau}

The approach of the present article is to aproximate the solution of a fractional
differential equation by that of the system \eqref{eq:ODEshort}.
Any code for the solution of stiff differential-algebraic
problems can be applied. We focus on our code
{\sc Radau5}, for which we develop a special treatment of the solution of linear
systems.

\subsection{Applying the code RADAU5}
\label{sect:applyradau}

The use of the code {\sc Radau5} is described in the monograph
\cite{hairer96sod}. For the problem \eqref{eq:ODEshort} one has to supply
a subroutine {\sc fcn(n,t,y,f,...)} for its right-hand side. Here, {\sc n}
is the dimension $d+n_1+\ldots +n_{d_I}$ of the system, {\sc y(n)} is the argument $Y(t)$,
and {\sc f(n)} the vector ${\cal F}\bigl(t,Y(t)\bigr)$.

Furthermore, a subroutine for the mass matrix $\cal M$ is needed.
One has to put {\sc imas = 1}, {\sc mlmas = 0}, {\sc mumas = 0}
in the calling program,
and one has to supply a subroutine {\sc mas(n,am,...)}, where {\sc am(1,j)}
contains the diagonal entries of ${\cal M}$.

One has the option to treat the arising linear system by unstructured full matrices
of dimension {\sc n}. An approximation
to the Jacobian matrix \eqref{eq:jacODE} can be computed internally by finite
differences by putting {\sc ijac = 0}, or one can put {\sc ijac = 1} and provide
a subroutine for \eqref{eq:jacODE}. Since $D = n_1+ \ldots + n_{d_I}$ is 
typically much larger than $d$, we recommend to exploit the pattern of the Jacobian matrix,
as we explain in the next section.

\subsection{Fast linear algebra}

A time integrator for the stiff differential equation \eqref{eq:ODEshort}
typically requires the solution of linear systems
$\bigl( (\gamma \Delta t)^{-1} {\cal M} - {\cal J} \bigr) u = a$, where
$u=(u_0, u_1 , \ldots , u_{d_I})$ and $a=(a_0, a_1 , \ldots , a_{d_I})$
are partitioned according to \eqref{eq:jacODE}.
Exploiting the arrow-like structure of \eqref{eq:jacODE} with diagonal blocks
(except for the first one, $J \in \R^{d \times d}$, and rank-$1$ off diagonal blocks,
it is convenient to solve the system backwards expressing $u_i$ in terms of $u_0$, which has linear complexity in the dimension of $u_i$, and finally 
computing $u_0$ as solution of a $d \times d$ linear system
\begin{equation}\label{eq:jhat}
\bigl( (\gamma \Delta t)^{-1} M - \hat J \bigr) u_0 = \hat a_0, \qquad\hat J
= J + \sum_{j=1}^{d_I} \hat c_j\, \frac{\partial F}{\partial I_j}
\frac{\partial G_j}{\partial y} ,
\end{equation}
where $\hat c_j$ is a scalar,
${\partial F}/{\partial I_j}$ is a column vector and
${\partial G_j}/{\partial y}$ is a row vector, so that their product gives
a rank-one matrix.

This allows to have a linear solver of complexity $\bigo (d^3 + D)$ instead of
$\bigo ((d+ D)^3)$ with $D = n_1+\ldots +n_{d_I}$ with typically $D \gg d$.
This is an essential feature of the approach we propose.
For details we refer to \cite[Section~3.1]{guglielmi25asi}, where closely related
ideas are discussed for a slightly different problem.

 \subsection{Use of the fast linear algebra (DC\_SUMEXP)}
 \label{sect:sumexp}

The code {\sc radau5} uses two files, {\sc decsol.f} for the linear algebra subroutines,
and {\sc dc\_decsol.f} for linking the linear algebra subroutines to the main
time integrator. For using a linear algebra that is adapted to the structure \eqref{eq:jacODE}
of the Jacobian matrix, we have written a subroutine {\sc dc\_sumexp.f}, which
replaces {\sc dc\_decsol.f}. We assume $d\ge 2$. For scalar problems 
(such as the example of Section~\ref{sect:example_1}) we recommend
to introduce at least one of the integral terms as  a new $y$-variable
(see Remark~\ref{rem:intd}).

When the subroutines of {\sc dc\_sumexp.f} are used, one has to define the dimension
of the system as \,{\sc n = }$d+(n_1+2)+\ldots +(n_{d_I} + 2)$.
The subroutines for the
right-hand side and the mass matrix are as in Section~\ref{sect:applyradau}. The
additional $2d_I$ variables are considered as dummy variables, and only used in the
subroutine for the Jacobian matrix. In the calling program one has to define
{\sc ijac = 1}, {\sc mljac = }$d$  (the dimension of the original system),
and {\sc mujac = }$0$.
One has to supply a subroutine {\sc jac(n,t,y,dfy,ldfy,...)}, where
the array {\sc dfy(ldfy,n)} ({\sc ldfy} is computed by {\sc radau5} and need not be
defined) contains all the information of the Jacobian matrix.

The upper left $d\times d$ matrix of {\sc dfy(ldfy,n)}
is that of the Jacobian matrix
\eqref{eq:jacODE}. Next to the right are $d_I$ blocks, each of them corresponding to
an integral term in the problem \eqref{eq:Voltgeneral}. The $j$th block has
$n_j+2$ columns. The first column contains the vector $\partial F/\partial I_j$,
and the second column contains the derivative $\partial G_j/\partial y$,
both are column vectors of dimension $d$.
The first row of this block
has the coefficients $c_{1,j},\ldots , c_{n_j,j}$ at the positions
$3$ to $n_j+2$. Similarly, the second row has $-\gamma_{1,j},\ldots , -\gamma_{n_j,j}$
at the positions $3$ to $n_j+2$. Finally, the last element of the third row is put to
a negative number to indicate the end of the $j$th block.

For the case that the dimension $d$ is large and the Jacobian gives raise
to banded matrices (this typically happens with space discretizations of
1D fractional PDEs) further options are
provided by the subroutines of {\sc dc\_sumexp.f}. This will be explained
in Section~\ref{sect:1DPDE}.

\section{Approximation of $t^{\alpha -1}$ by a sum of exponential functions}
\label{sect:approxsumexp}

This section is devoted to the development of explicit formulas of the coefficients
\eqref{eq:sum-exp} for the case $0<\alpha <1$. Among the different formulations
of fractional differential equations this covers
\eqref{Volt2} for every $\alpha >0$, and the formulation
\eqref{Volt1} for the case $0<\alpha < 1$.

\subsection{Approach of Beylkin and Monz\' on}
\label{sect:beylkin-monzon}

Since the Laplace transform of the function $z^{-\alpha }$ is $\Gamma (1-\alpha )\, t^{\alpha -1 }$
for $\alpha < 1$,
we have the integral re\-pre\-sen\-ta\-ti\-on
\begin{equation}\label{integral-repres}
\frac {t^{\alpha -1}}{\Gamma (\alpha )} = \frac{\sin (\pi\alpha)}{\pi} \int_0^\infty \e^{-tz}z^{-\alpha } \,\d z =
 \frac{\sin (\pi\alpha)}{\pi} \int_{-\infty}^\infty \e^{-t \e^s}\e^{(1-\alpha ) s} \,\d s ,
\end{equation}
where we have used the identity $\Gamma (1-\alpha )\Gamma (\alpha ) = \pi /\sin (\pi\alpha )$.
To express this function as a sum of exponentials, \cite{beylkin10abe} proposes to
approximate the integral to the right by the trapezoidal rule. With a step size $h>0$ this yields
\begin{equation}\label{trapez-bm}
T (t,h) = h\,\frac{ \sin (\pi\alpha)}{\pi} 
\sum_{i=-\infty }^\infty \e^{(1-\alpha ) ih} \e^{-\e^{ih} t} .
\end{equation}

\noindent
{\bf Error of the trapezoidal rule.}
It follows from \cite[Theorem 5.1]{trefethen14tec},
in the same way as in \cite{guglielmi25asi},
that the error due to this approximation
can be bounded by
\begin{equation}\label{est-T-bm-inf}
\bigg| \frac{t^{\alpha -1}}{\Gamma (\alpha )} - T(t,h) \bigg| \le 
c_\alpha \,\frac{t^{\alpha -1}}{\Gamma (\alpha )}, \qquad
c_\alpha = \frac{2 (\cos a)^{\alpha -1}}{\e^{2\pi a /h} -1} 
\end{equation}
for every $h>0$ and any $0<a< \pi /2$. Choosing $a$ close to optimal,
we obtain $c_\alpha \le \eps$, if the step size $h$ satisfies the inequality in
\begin{equation}\label{stepsize-bm}
h\le \frac{2\pi a}{\ln \bigl( 1 + \frac 2\eps (\cos a )^{\alpha -1} \bigr) } , \qquad 
a = \frac \pi 2 \Big( 1 - \frac {(1- \alpha ) }{(2 - \alpha )\, \ln \eps^{-1}} \Big) .
\end{equation}

\noindent
{\bf Error from truncating the series \eqref{trapez-bm}.}
Bounding the sum by an integral and using the definition 
$\,\Gamma (\alpha , x) = \int_x^\infty \e^{-\sigma } 
\sigma^{\alpha -1}\, \d \sigma \,$ of the incomplete Gamma function,
we obtain (see also \cite{guglielmi25asi})
\begin{equation}\label{err-trunc}
\begin{array}{rcl}
 \displaystyle
 h\,\frac{ \sin (\pi\alpha)}{\pi} \sum_{i=-\infty }^{M-1} \! \e^{\alpha ih} \e^{-\e^{ih} t} 
&\le &\displaystyle \frac{t^{\alpha -1 }}{\Gamma (\alpha )} \biggl(1 - 
\frac{\Gamma (1-\alpha ,t \e^{Mh})}{\Gamma (1-\alpha )}\biggr) ,\\[4mm]
 \displaystyle
 h\,\frac{ \sin (\pi\alpha)}{\pi} \sum_{i=N}^{\infty} \! \e^{\alpha ih} \e^{-\e^{ih} t} 
&\le & \displaystyle\frac{t^{\alpha -1 }}{\Gamma (\alpha )} 
\biggl(\frac{\Gamma (1-\alpha ,t \e^{Nh})}{\Gamma (1-\alpha )}\biggr) .
\end{array}
\end{equation}
Since $\Gamma (1-\alpha , 0 ) = \Gamma (1-\alpha )$ and
$\Gamma (1-\alpha , \infty )=0$, the bounds can be made arbitrarily small
for fixed $ t > 0$
by choosing $M$ large negative, and $N$ large positive.
\medskip

\noindent
{\bf Choice of the parameters.}
For a given accuracy requirement $\eps$ we first choose $h$ as large as possible
satisfying \eqref{stepsize-bm}. The bounds \eqref{err-trunc} cannot be made arbitrarily small when $t \to 0$ or $t \to \infty$. We therefore restrict our estimates
to an interval $\delta \le t \le T$ with $\delta >0$ and $T< \infty$.
We then choose $M$ a large negative number such that
$\Gamma (1-\alpha ,T \e^{Mh}) \ge (1- \eps )\Gamma (1-\alpha )$, and finally we choose
$N$ a large positive number such that
$\Gamma (1-\alpha ,\delta \e^{Nh}) \le \eps \Gamma (1-\alpha )$.
This choice of the parameters implies that
\begin{equation}\label{est-T-bm}
\bigg| \frac{t^{\alpha -1}}{\Gamma (\alpha )} - T_M^N(t,h) \bigg| \le 
3 \,\eps \,\frac{t^{\alpha -1}}{\Gamma (\alpha )}, \qquad
T_M^N (t,h) = h\,\frac{ \sin (\pi\alpha)}{\pi} 
\sum_{i=M }^{N-1}\e^{(1-\alpha ) ih} \e^{-\e^{ih} t} 
\end{equation}
on the interval $\delta \le t \le T$. If we are interested to solve the
fractional differential equation on the interval $[0,T]$, we obviously let
$T$ in the choice of $N$ be the endpoint of integration.
The choice of $\delta > 0$ (und thus of $M$) will be discussed
in Section~\ref{sect:comp-sumexp}.

\subsection{Effect of approximating the kernel}
\label{sect:effect}

We are interested to study the influence of an approximation of the kernel
to the solution of the original problem. For this we consider the
fractional differential equation \eqref{eq:fde} with
$0< \alpha < 1$, written in its equivalent form
\begin{equation} \label{Volt1-m1}
y(t) = y_0 + \frac{1}{\Gamma(\alpha)}\int_0^{t} (t-s)^{\alpha-1}
f\big(s,y(s)\big)\,ds.
\end{equation}
We let $\tilde y(t)$ be the solution of \eqref{Volt1-m1},
where $k(t) = t^{\alpha -1}/\Gamma (\alpha )$ is
replaced by an approximation $\tilde k(t)$.
We are interested in the error $\| y(t) - \tilde y(t) \|$.
For this we assume Lipschitz continuity
\begin{equation}\label{lipschitz}
\big\| f(t,y) - f(t,\tilde y )\big\| \le L_f \big\|y - \tilde y\big\| 
\end{equation}
as well as
$\|f(t,y)\|\le M_f$ in a neighbourhood of the solution.

\begin{thm}\label{thm:approx}
Consider the problem \eqref{Volt1-m1}, 
and assume that the approximation $\tilde k(t)$ of the kernel
$k(t) = t^{\alpha -1}/\Gamma (\alpha )$
(with $0 < \alpha < 1$)
satisfies
\begin{equation}\label{kernel-approx}
\int_0^t \Big| \frac{s^{\alpha -1}}{\Gamma (\alpha )} - \tilde k (s) \Big| \, \d s \le \eps \bigl( 1 + t^\alpha \bigr). 
\end{equation}
Then, we have
\begin{equation}\label{estimate}
\big\| y(t) - \tilde y (t) \big\| \le \eps\, u(t) ,
\end{equation}
where $u(t)$ is solution of the scalar equation
\begin{equation}\label{solu-u}
u(t) = \frac{L_f}{\Gamma (\alpha )} \int_0^t (t-s)^{\alpha -1 } u(s) \, \d s 
+  M_f \bigl( 1 + t^\alpha \bigr) .
\end{equation}
\end{thm}

\begin{proof}
Subtracting the equation for $\tilde y (t)$ from \eqref{Volt1-m1} and
using the Lipschitz condition and boundedness of $f(t,y)$ yields
\begin{equation} \label{inequforu}
\| y(t) - {\tilde y} (t) \| \le
\frac{L_f}{\Gamma (\alpha )} \int_0^t (t-s)^{\alpha -1} \| y(s) - {\tilde y} (s) \|\, \d s +
M_f \int_0^t \Big| \frac{s^{\alpha -1}}{\Gamma (\alpha )} - \tilde k (s) \Big| \, \d s .
\end{equation}
The assumption on the approximation $\tilde k (t)$ then implies
$\|y(t) - {\tilde y} (t) \| \le \eps u(t)$ with $u(t)$ given
by \eqref{solu-u}.
\end{proof}

\begin{remark}\rm
The scalar equation \eqref{solu-u} can be solved by using Laplace transforms.
For the function $u(t)$ we denote the Laplace transform by $\widehat u (w)$,
and we recall that the Laplace transform of 
$t^{\alpha-1}$ ($\alpha > 0$) is 
$\Gamma(\alpha ) w^{-\alpha}$.
Passing to Laplace transforms in \eqref{solu-u} we get
\begin{equation} \nonumber
\widehat{u}(w)  =  \frac{L_f}{ w^{\alpha} } \,\widehat{u}(w)
+   \frac{M_f}{w}
\biggl( 1 + \frac{\Gamma (\alpha +1)}{w^{\alpha}} \biggr) ,
\end{equation}
so that
\begin{equation} \nonumber
\widehat{u}(w)  =  
   \frac{M_f}{w}\biggl( 1 - \frac{L_f}{ w^{\alpha} }\biggr)^{-1}
\biggl( 1 + \frac{\Gamma (\alpha +1)}{w^{\alpha}} \biggr) .
\end{equation}
Taking the inverse Laplace transform gives the solution of \eqref{solu-u}.

For example, for $\alpha = 1/2$, this gives
\[
u(t) = \frac{ M_f}{2 L_f} \left(2 L_f+\sqrt{\pi }\right)
   \e^{L_f^2 t} \left(\text{erf}\left(L_f
   \sqrt{t}\right)+1\right).
\]
In general we expect exponential bounds of this kind, similarly as for
ordinary differential equations.
\end{remark}

\subsection{Computing the approximation by a sum of exponentials}
\label{sect:comp-sumexp}

To get a good approximation of the solution, the kernel approximation has to
satisfy the condition \eqref{kernel-approx} of Theorem~\ref{thm:approx}.
In Section \ref{sect:beylkin-monzon} we have explained how the parameters $h$,
$M$, and $N$ have to be chosen to get an approximation $T_M^N (t,h)$
satisfying \eqref{est-T-bm}. This implies, for $\delta \le t \le T$,
\begin{equation}\label{kernel-approx1}
\int_\delta^t \Big| \frac{s^{\alpha -1}}{\Gamma (\alpha )} - 
T_M^N (s,h) \Big| \, \d s \le 3\,\eps\, \frac{t^\alpha}{\Gamma (\alpha + 1)}  ,
\end{equation}
which, up to a constant, is compatible with \eqref{kernel-approx}.
We still have to find a $\delta$, such that the integral over $(0,\delta )$
is $\bigo (\eps )$.
The estimates \eqref{est-T-bm-inf} and \eqref{err-trunc} are valid for all $t>0$.
For $t\in (0,\delta )$ they give bounds $\eps \,t^{\alpha -1}/\Gamma (\alpha )$,
with the exception of the second estimate of \eqref{err-trunc},
which is only bounded by $t^{\alpha -1}/\Gamma (\alpha )$.
To satisfy \eqref{kernel-approx},
we choose $\delta > 0$ such that 
\begin{equation}
\int_0^\delta \frac{t^{\alpha -1}}{\Gamma (\alpha )}\, \d t = 
\frac{\delta^\alpha}{\Gamma(\alpha +1)}  \le \eps. 
\end{equation}
For the computation of $M$ and $N$, we have to find $x_* = T \e^{M h}$ and $x^* = \delta \e^{N h}$ such that 
\begin{equation*}
\Gamma(1 - \alpha, x_*) \ge (1 - \eps) \Gamma(1-\alpha) \quad \mbox{and} \quad 
\Gamma(1 - \alpha, x^*) \le \eps \Gamma(1-\alpha).
\end{equation*}
Making use of the asymptotic formulas
$\Gamma(1 - \alpha, x) \ge \Gamma(1-\alpha) - {x^{1-\alpha}}/(1-\alpha)$ and 
$\Gamma(1 - \alpha, x) \le  x^{-\alpha}\,\e^{-x},$
we find the approximations $x_* =  \bigl(\Gamma (2-\alpha ) \, \eps \bigl)^{1/(1-\alpha)}$ and $x^* = - \ln \bigl( \Gamma (1-\alpha )\, \eps \bigr)$.

For a given parameter $\alpha\in (0,1)$,
an accuracy requirement $\eps$ and a final time $T$,
Algorithm~\ref{alg_para} summarizes the computation of the parameter
$\delta$ as well $ h, M, N$, such that the kernel approximation
$T_M^N (t,h)$ of \eqref{est-T-bm} satisfies \eqref{kernel-approx}.

\medskip
\begin{algorithm}[H] \label{alg_para}
\DontPrintSemicolon
\KwData{$\alpha$, $\eps$, $T$}
\KwResult{$\delta, h,M, N$} 
\Begin{
\nl Compute $\delta= \bigl( \Gamma(\alpha +1 ) \,\eps \bigr)^{1/\alpha }$\;
\nl Set $h = \displaystyle \frac{2\pi a}{\ln \bigl( 1 + \frac 2\eps (\cos a )^{\alpha -1} \bigr) }$, where $\displaystyle
a = \frac \pi 2 \Big( 1 - \frac {(1- \alpha ) }{(2 - \alpha )\, \ln \eps^{-1}} \Big)$\;
\nl Set $x_* =  \bigl(\Gamma (2-\alpha ) \, \eps \bigl)^{1/(1-\alpha)}$ and
compute $M$ such that $T\e^{M h} = x_*$\;
\nl Set $x^* = - \ln \bigl( \Gamma (1-\alpha )\, \eps \bigr)$ and
compute $N$ such that $\delta \e^{N h}  = x^*$\;
\Return
}
\caption{\rule{0mm}{3mm}Choice of the parameters for the
fractional kernel $t^{\alpha -1}/\Gamma (\alpha )$.}
\end{algorithm} 
\bigskip

To get an impression of the values for $h$, $M$, and $N$, we present
in Table~\ref{tab:param5} their values for the choice $T=1000$,
for $\eps = 10^{-5}$ and
$\eps = 10^{-10}$, and for different values of $\alpha \in (0,1)$.

\begin{table}[ht] 
\begin{center}
\caption{Parameters for the fractional kernel for
$T=1000$. }
\label{tab:param5}
{\small
\begin{tabular}{|c|c|r|r|r|r|r|r|r|r|r|}
\hline
 \rule[0mm]{0cm}{3.6mm}%
 &$\alpha$  & $0.1$ &  $0.2$ &  
$0.3$ & $0.4$ & $0.5$ & $0.6$ &
$0.7$ & $0.8$ & $0.9$\\
\hline
 \rule[0mm]{0cm}{3.0mm}%
&$h$ &  $0.65$ &  $0.66$ & 
$0.67$ & $0.69$ & $0.70$ & $0.72$ &
$0.73$ & $0.75$ & $0.78$\\
$\eps= 10^{-5}$ &$M$ &  $-31$ &   $-33$ & 
$-36$ & $-39$ & $-44$ & $-51$ &
$-63$ & $-87$ & $-159$\\
&$N$ &  $148$ &   $93$ &  
$62$ & $47$ & $37$ & $31$ &
$26$ & $23$ & $20$\\
  \hline
 \rule[0mm]{0cm}{3.0mm}%
&$h$ &  $0.37$ &  $0.37$ & 
$0.38$ & $0.38$ & $0.39$ & $0.39$ &
$0.40$ & $0.40$ & $0.41$\\
$\eps = 10^{-10}$ &$M$ &  $-91$ &   $-99$ & 
$-109$ & $-122$ & $-141$ & $-169$ &
$-215$ & $-308$ & $-586$\\
&$N$ &  $649$ &   $326$ &  
$218$ & $163$ & $131$ & $109$ &
$93$ & $81$ & $71$\\
  \hline
\end{tabular}
}
\end{center}
\end{table}

\section{The case $\alpha > 1$}
\label{sect:alphage1}

Until now we mainly treated the situation, where $0<\alpha <1$.
Here we consider an integral term \eqref{integral} with $\alpha >1$.
One possibility is to introduce
the new variable
\begin{equation*}
y_{d+1}(t) = \frac 1{\Gamma (\alpha )} \int_0^t (t-s)^{\alpha -1} 
G\big(t,y(s)\bigr)\, \d s .
\end{equation*}
For $1<\alpha < 2$, we differentiate it with respect to time and obtain
the differential equation
\[
\dot y_{d+1} (t) = \frac 1{\Gamma (\alpha -1 )} \int_0^t (t-s)^{\alpha -2}
G\big(t,y(s)\bigr)\, \d s , \qquad y_{d+1} (0) = 0 ,
\]
which we add to the original system.
Since $\alpha -2 = \alpha_0 -1$ with $\alpha_0  \in (0,1)$, we are
in the situation that can be treated by the approach
of Section~\ref{sect:beylkin-monzon}.
For an $\alpha \in (2,3)$, we differentiate twice, and so on.
With this transformation we can assume without loss
of generality that in the problem \eqref{eq:fde} all integral
terms have $\alpha_j \in (0,1)$. The present section is devoted
to another possibility for treating the case $\alpha > 1$.

\subsection{Splitting of the kernel}
\label{sect:splitting}

The kernel approximation of Section~\ref{sect:approxsumexp} is
valid only for $0<\alpha < 1$.
For $\alpha > 1$, $m= \lceil \alpha \rceil \ge 2 $, we put
$\alpha _0 = \alpha -m+1 \in (0,1)$. It is possible to
split the kernel according to
\begin{equation}\label{eq:split}
\frac{t^{\alpha -1 }}{\Gamma (\alpha )} = 
\frac{t^{m-1}}{(\alpha -1)\cdots (\alpha -m +1)} \cdot 
\frac{t^{\alpha_0 -1}}{\Gamma (\alpha_0 )} ,
\end{equation}
and to apply the approximation \eqref{est-T-bm} to the factor
$t^{\alpha_0 -1}/\Gamma (\alpha _0)$ (see also \cite{guglielmi25asi}).
This then gives an approximation for $t^{\alpha -1}/\Gamma (\alpha )$
as a sum of exponentials multiplied by a monomial. The linear chain
trick can be extended to this situation, which also reduces the integral
equation to a system of ordinary differential equations.

\subsection{Extension of the linear chain trick}

For notational convenience we here suppress the index $j$ in the
integral \eqref{integral}, and also in $\alpha_j$ and $G_j (t,y)$.
We assume that the second factor in \eqref{eq:split} is approximated by
\begin{equation}\label{eq:sum-exp0}
\frac 1{\Gamma (\alpha_0 )}\,t^{\alpha_0 -1} \approx 
\sum_{i=1}^n c_i \,\e^{-\gamma_i t} .
\end{equation}
Since $0<\alpha_0 < 1$, this approximation can be the one of
Section~\ref{sect:beylkin-monzon}.
We now replace the integral \eqref{integral} by
\begin{equation}\label{intappro}
\frac 1{\Gamma (\alpha )}\int_0^t (t-s)^{\alpha -1} G\bigl( s, y(s)\bigr)\, \d s
\approx \frac{1}{(\alpha -1)\cdots (\alpha -m +1)}
\sum_{i=1}^n c_i \, z_{i,m}(t) ,
\end{equation}
where $z_{i,m}(t)$ is the last function among
\[
z_{i,k} (t) = \int_0^t (t-s)^{k -1} 
\e^{-\gamma_i (t-s)}G\bigl( s, y(s)\bigr)\, \d s , \qquad
k=1,\ldots , m .
\]
Differentiation of this function gives the ordinary differential equations
\begin{equation}\label{eq:sys-zik}
\dot z_{i,k} (t) = -\gamma_i \,z_{i,k}(t) + \biggl\{\begin{array}{cl}
    (k-1)\, z_{i,k-1}(t) & ~~ k=2,\ldots ,m \\[1.5mm]
    G\bigl( t,y(t)\bigr) & ~~ k=1.
\end{array}\bigg.
\end{equation}
If the integral terms in \eqref{eq:fde} with $\alpha_j > 1$ are replaced
by the approximation \eqref{intappro}, we have to add the $m$ differential
equations \eqref{eq:sys-zik} to the original system.

\subsection{Structure of the Jacobian matrix}

The Jacobian matrix of the complete system is again of the form
\eqref{eq:jacODE} with a different interpretation of $J_j$, $B_j$,
and $C_j$ for those indices $j$, for which $\alpha_j > 1$.

The matrix $J_j$ is bi-diagonal of dimension $m_j \cdot n_j$, where
$m_j = \lceil \alpha_j\rceil$ and $n_j$ is the number of
summands in \eqref{eq:sum-exp0}. The diagonal elements are $m_j$ times
$\gamma_{1,j}$, then $m_j$ times
$\gamma_{2,j}$ until a block of $m_j$ elements equal to $\gamma_{n_j,j}$.
The subdiagonal
consists of $n_j$ vectors $(1,\ldots ,m_j-1)$ separated by $0$.
This means that for $1<\alpha_j <2$ the subdiagonal is
$(1,0,1,0,\ldots ,1,0)$.

The matrix $B_j = e (\partial G_j /\partial y)$ is a rank-one matrix,
where the derivative $\partial G_j /\partial y$ is a row vector of dimension $d$.
The column vector $e$ is composed of $n_j$ vectors $(1,0,\ldots , 0)^\top
\in \real^{m_j}$.

The matrix $C_j = (\partial F/\partial I_j) \tilde c_j^\top$
is also a rank-one matrix. Here,
the vector
$\tilde c_j^\top$ is composed of $n_j$ vectors
$(0,\ldots ,0,c_{i,j})\in \real^{m_j}$
(for $i=1,\ldots , n_j$).

\subsection{Use of the fast linear algebra in RADAU5}
\label{sect:sumexp2}

Similar as in Section~\ref{sect:sumexp} we can use the subroutines
of {\sc dc\_sumexp.f} for a fast linear algebra. Here, the dimension
of the system has to be defined as
$N=d  + (m_1n_1 +2) + \ldots + (m_{d_I}n_{d_I} +2)$.
The parameters
{\sc ijac = 1}, {\sc mljac = }$d$, and
and {\sc mujac = }$0$ are the same as in Section~\ref{sect:sumexp}.
Also the upper left $d\times d$ matrix of {\sc dfy(ldfy,n)}
is that of the Jacobian matrix \eqref{eq:jacODE}. 

Next to the right we still have $d_I$ blocks corresponding to
integral terms in the problem \eqref{eq:Voltgeneral}. Here, the $j$th block has
$m_jn_j+2$ columns. 
As in Section~\ref{sect:sumexp}
the first column contains the vector $\partial F/\partial I_j$,
and the second column contains the derivative $\partial G_j/\partial y$, written as
column vector.

The first row of the $j$th block
has the coefficients of $\tilde c_j$ at the positions
$3$ to $m_jn_j+2$, and the second row has the diagonal of $J_j$
at the positions $3$ to $m_jn_j+2$. New is that the third row contains
the subdiagonal elements of $J_j$.
The last element of the third row, which is not used by the subdiagonal, is put to
a negative number to indicate the end of the $j$th block.

\section{Numerical experiments}
\label{sect:numerical}

We consider a few illustrative examples. The first one is scalar
with known exact solution. The second one is a modification of the
well-known Brusselator equation. Both of them are used as test
equations in a series of publications. The third and fourth are fractional
derivative partial differential equations.

\subsection{Example 1: scalar fractional ODE with exact solution}
\label{sect:example_1}

As a first test example we consider the equation
from Diethelm et al.~\cite{Die04} 
\begin{equation}\label{example1}
D_*^{\alpha}y(t) = 
\frac {9\,\Gamma (1+\alpha)}4 
-\frac{3\,
   t^{4-\alpha /2} \,\Gamma \left(5+\alpha /2\right)}
   {\Gamma \left(5-\alpha /2\right)}+ \frac{\Gamma (9) \,t^{8-\alpha }}{\Gamma (9-\alpha )}
   +\Bigl(\frac{3}{2} t^{\alpha/2}-t^4 \Bigr)^3 -y(t)^{3/2},
\end{equation}
with $y(0)=0$. The
inhomogeneity is chosen such that
\[
y(t)= \frac{9\, t^{\alpha }}{4} -3 \,t^{4+\alpha/2}+t^8 =
\Bigl(\frac{3}{2} t^{\alpha/2}-t^4 \Bigr)^2 
\]
is the exact solution of the equation. It starts at $y(0)=0$, reaches a maximum
at $t= (3\alpha /16)^{1/(4-\alpha /2)}$, and
vanishes at
$t=(3/2)^{1/(4-\alpha /2)}$ which, e.g.\ for $\alpha = 1/2$, results
in a loss of numerical well-posedness of the problem for $t > 1$.
Hence, in our numerical experiments
we integrate the equation only up to $T=1$.
\medskip

\noindent{\it
First experiment (connection between ${\it Tol}$ and $\eps$).}
For the numerical integration we apply the code {\sc Radau5}
with accuracy requirement ${\it Atol} = {\it Rtol} = {\it Tol} =10^{-7}$
to the augmented system \eqref{eq:ODE}.
We use the kernel approximation \eqref{est-T-bm}
with parameters $h,M,N$, determined by Algorithm~\ref{alg_para}.
We study the relative error at $T=1$ for different values of $\eps$,
which determines the accuracy of the approximation by the sum of exponentials.

\begin{table}[th!]
\begin{center}
{\small
\begin{tabular}{|c||c|c|c|c|c|} \hline
 $\eps$ & $h$ & $\delta$ & $M$ & $N$ & ${\rm err}$ \\ 
\hline
\hline
\rule[0mm]{0cm}{3.5mm}%
 $10^{-4}$  &  $ 0.839$  &  $7.85 \cdot 10^{-9}$  &  $ -23$  &  $  25$  &  $6.35 \cdot 10^{-5}$ \\
 $10^{-5}$  &  $ 0.697$  &  $7.85 \cdot 10^{-11}$ &  $ -34$  &  $  37$  &  $6.36 \cdot 10^{-6}$ \\
 $10^{-6}$  &  $ 0.596$   & $7.85 \cdot 10^{-13}$ &  $ -47$  &  $  52$  &  $5.77 \cdot 10^{-7}$ \\
 $\mathbf{10^{-7}}$      &  $\mathbf{0.522}$      &  $\mathbf{7.85 \cdot 10^{-15}}$  &  $\mathbf{-63}$  &  
 $\mathbf{68}$  &  $\mathbf{ 5.63 \cdot 10^{-7}}$ \\
 $10^{-8}$  &  $ 0.469$  &  $7.85 \cdot 10^{-17}$  &  $-80$   &  $ 87$   &  $6.37 \cdot 10^{-7}$ \\
 $10^{-9}$  &  $ 0.418$  &  $7.85 \cdot 10^{-19}$  &  $-100$  &  $ 108$  &  $7.23 \cdot 10^{-7}$ \\
 $10^{-10}$ &  $ 0.380$  &   $7.85 \cdot 10^{-21}$  &  $-122$  &  $ 131$  &  $5.79 \cdot 10^{-7}$ \\
  \hline
\end{tabular}
}
\vskip 1.5mm
\caption{Error behavior obtained by applying {\sc Radau5} to the test problem \eqref{example1} 
with final point $T=1$ and ${\it Tol}= 10^{-7}$ for different computed values of $\delta$, 
$h$, $M$, and $N$ (computed by Algorithm \ref{alg_para}).}
\label{tab:11} 
\vspace{-5mm}
\end{center}
\end{table} 

The result is shown in Table~\ref{tab:11} for $\alpha = 1/2$.
We can observe that for $\eps\ge {\it Tol}$ the error is essentially proportional to
$\eps$, which corresponds to the error of the fractional kernel approximation. For
$\eps \le {\it Tol}$ the error remains close to ${\it Tol}$, which indicates that the
error of the time integration is dominant.
The parameter values for $\eps = {\it Tol} =10^{-7}$
are indicated in bold.
\medskip

\noindent{\it
Second experiment.}
As a second experiment we look at the linear algebra. 
We apply both the standard Gauss method and the fast algorithm exploiting the structure of the Jacobian. The results are shown in Table \ref{tab:12-eh} and show a striking improvement of the CPU time due to the fast linear solver.
Since both approaches are mathematically equivalent, they give the same
error in the solution approximation.

\begin{table}[th!] 
\begin{center}
{\small
\begin{tabular}{|c|c|c|c|c|}
\hline
 \rule[0mm]{0cm}{3.6mm}%
 ${\it Tol} = \eps $  & $10^{-5}$ &  
$10^{-7}$ & $10^{-9}$ & $10^{-11}$\\
\hline
 \rule[0mm]{0cm}{3.5mm}%
cpu (standard) &  $0.36 \cdot 10^{-1}$ &
$0.17\cdot 10^{0}$ & $0.69\cdot 10^{0}$ 
& $0.25\cdot 10^{1}$ \\
cpu (fast) &  $0.96\cdot 10^{-3}$ &
$0.21\cdot 10^{-2}$ &  $0.58\cdot 10^{-2}$ & $0.16\cdot 10^{-1}$ \\
${\rm err}$ &  $1.4\cdot 10^{-5}$ &
$5.63 \cdot 10^{-7}$ & $ 2.62\cdot 10^{-8}$ & $5.50\cdot 10^{-10}$ \\
  \hline
\end{tabular}
\vspace{1.5mm}
\caption{Comparison of different linear algebra solvers for the system \eqref{example1} with $\alpha = 1/2$.}
\label{tab:12-eh} 
}
\vspace{-5mm}
\end{center}
\end{table} 

\medskip
\noindent{\it
Third experiment (comparison of different integral formulations).}
For the case where $\alpha > 1 $, one can either use the formulation
\eqref{Volt1} as a Volterra integral equation, or the
formulation \eqref{Volt2} as a Volterra integro-differential equation.
Table~\ref{tab:13} shows, for fixed ${\it Tol} = \eps = 10^{-6}$
the values of $M$, $N$ in the approximation by a sum of exponentials
and the error and cpu time as a  function of $\alpha$, when $\alpha$
ranges from $1$ to $2$. The value of $M$ is the same for both
approaches and also the error is similar.
In particular, for $\alpha$ close to~$1$, the value of
$N$ is significantly larger for the version \eqref{Volt2}.
When looking at the cpu time, one notices that the formulation
\eqref{Volt1} is more efficient for $\alpha$ close to~$1$, whereas
the formulation \eqref{Volt2} is slightly better for 
$\alpha$ close to $2$.

\begin{table}[h] 
\begin{center}
{\small
\begin{tabular}{|c|c|c|c|c|c|}
\hline
 \rule[0mm]{0cm}{3.1mm}%
$\alpha$ &  $1.1$ & $1.3$ & $1.5$ & $1.7$ & $1.9$ \\
\hline
\hline
 \rule[0mm]{0cm}{3.1mm}%
$M$ &  $-28$ & $-35$ & $-47$ & $-75$ & $-212$ \\
$N$ &  $28$ & $23$ & $20$ & $17$ & $15$ \\
${\rm err}$ & 
$0.33 \cdot 10^{-6}$ & $ 0.74\cdot 10^{-6}$ & $0.14\cdot 10^{-5}$
& $0.11\cdot 10^{-5}$ & $0.77\cdot 10^{-6}$ \\
${\rm cpu~time}$ &  $0.11\cdot 10^{-2}$ & $0.10 \cdot 10^{-2}$
& $ 0.11\cdot 10^{-2}$ & $0.11\cdot 10^{-2}$ & $0.28\cdot 10^{-2}$ \\
  \hline
\hline
 \rule[0mm]{0cm}{3.1mm}%
$M$ &  $-28$ & $-35$ & $-47$ & $-75$ & $-212$ \\
$N$ &  $235$ & $86$ & $52$ & $37$ & $28$ \\
${\rm err}$ & 
$0.25 \cdot 10^{-5}$ & $ 0.11\cdot 10^{-5}$ & $0.44\cdot 10^{-7}$
& $0.44\cdot 10^{-6}$ & $0.57\cdot 10^{-6}$ \\
${\rm cpu~time}$ &  $0.61\cdot 10^{-2}$ & $0.15 \cdot 10^{-2}$
& $ 0.89\cdot 10^{-3} $ & $0.11\cdot 10^{-2}$ & $0.17\cdot 10^{-2}$ \\
  \hline
\end{tabular}
\vspace{1.5mm}
\caption{The upper block shows the values $M$, $N$, the error and the
cpu time for the formulation \eqref{Volt1} of the
problem \eqref{example1}, whereas the lower block shows them
for the formulation \eqref{Volt2}.}
\label{tab:13} 
}
\vspace{-5mm}
\end{center}
\end{table}

\subsection{Example 2: Brusselator model}
\label{sect:example_2}

The so-called {\it Brusselator} is a model of a chemical reaction
\cite{lefever71cia} that possesses periodic solutions and has applications
in the interpretation of biological phenomena. It is often used as a
test example for nonstiff time integrators. A modification, where the
time derivative is replaced by a fractional derivative, is proposed by
\cite{Gar18} as a test for codes solving fractional differential equations.
It is given by
\begin{equation}\label{example2}
\begin{array}{rcl}
D_*^{\alpha_1}y_1(t) & = & 
A - (B+1) y_1(t) + y_1(t)^2 y_2(t)
\\[2mm]
D_*^{\alpha_2}y_2(t) & = & 
B y_1(t) -  y_1(t)^2 y_2(t)
\end{array}
\end{equation}
We choose the parameters and initial values as $A=1$, $B=3$,
and $y_1(0)=1.2$, $y_2(0)=2.8$. The degrees of the derivatives are
$\alpha_1 = 1.3$, $\alpha_2=0.8$. Since $\alpha_1 >1$, we need 
a further initial value $\dot y_1 (0) =1$.
Figure~\ref{fig:Brus} shows the two solution components as a
function of time.

\begin{figure}[H]
\begin{center}
\begin{picture}(341.2,136.3)(0.,0.)
\epsfig{file=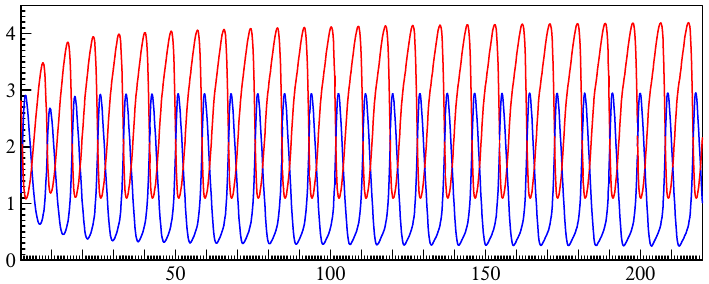}
\end{picture}
\vspace{-2mm}
 \caption{Solution of \eqref{example2} with parameter $\alpha_1=1.3, \alpha_2=0.8$ in the time interval $[0,220]$ ($y_1(t)$ is in blue, $y_2(t)$ is in red) \label{fig:Brus}}
 \end{center}
\end{figure}

We apply the code {\sc radau5} to this fractional differential equation
on the interval $[0,220]$ with different values of ${\it Tol} = \eps$.
For the first component (with $\alpha_1 > 1$) we use the approach
\eqref{Volt2}. In Table~\ref{tab:21} we present the relative error
and the cpu time. In the first two rows are the values $M$ and $N$
obtained by Algorithm~\ref{alg_para}. Each entry has two numbers:
the first one corresponds to $\alpha_1$, the second to $\alpha_2$.

\begin{table}[ht!] 
\begin{center}
{\small
\begin{tabular}{|c|c|c|c|c|}
\hline
 \rule[0mm]{0cm}{3.6mm}%
 ${\it Tol} = \eps $  & $10^{-4}$ & 
$10^{-6}$ & $10^{-8}$ & $10^{-10}$\\
\hline
 \rule[0mm]{0cm}{3.0mm}%
$M$ &  $-24~~-57$ &   $-44~~-118$ & $-71~~-200$ & $-104~~-304$ \\
$N$ &  $42~~~~15$ & $86~~~~32$ &  $144~~~~53$ & $218~~~~81$ \\
cpu (fast) &  $0.10\cdot 10^{-1}$ &  $0.51\cdot 10^{-1}$  
& $0.10\cdot 10^{0}$ & $0.31\cdot 10^{0}$  \\
${\rm err}$ &  $0.69\cdot 10^{-2}$ &   $0.60 \cdot 10^{-4}$  
& $0.67 \cdot 10^{-6}$ & $0.89 \cdot 10^{-8}$  \\
  \hline
\end{tabular}
\vspace{1.0mm}
\caption{Numerical integration of \eqref{example2} using the fast linear algebra routine}.
\label{tab:21} 
}
\end{center}
\vspace{-6mm}
\end{table} 

An accurate approximation of the solution at $T=220$ is given by
\[
y_1(T) = 1.0097684171, \quad
y_2(T) = 2.1581264031 .
\]
We note that with a value $\eps=\tol=10^{-6}$
we obtain a relative error $0.60 \cdot 10^{-4}$ in a
cpu time of $0.051$ seconds. The computation takes $1244$
accepted steps, which corresponds to a mean step size $h\approx 0.18$.

\subsubsection*{Comparison with the codes by Garrappa \cite{Gar18}}
We test the problem with the codes at
\url{https://www.dm.uniba.it/it/members/garrappa/software}.
For these codes the user has to provide a step size instead of
an accuracy requirement ${\it Tol}$.

Applying the implicit code {\sc fde\_pi1\_im} with step size
$h=10^{-4}$ gives a relative error ${\rm err} = 2.0665 \cdot 10^{-4}$
in about $84$ seconds.
Instead the explicit code {\sc fde\_pi1\_ex} yields a relative error 
${\rm err} = 2.0696 \cdot 10^{-4}$ in about $63.75$ seconds. 

The peak of memory allocation is about $165$ Megabytes for the implicit code
and $164$ Megabytes for the explicit one. With larger $T$ and smaller
step size $h$, these codes may encounter memory problems.

The difference with respect to our integrator is significant.
In order to get the same relative error, our code (using a variable step size) requires a CPU time of less than $0.05$ seconds. Since our code solves an ordinary differential
equation, no storage of the solution in the past is needed.
Therefore, we do not have any difficulties with the memory requirement.

\subsubsection*{Comparison with the codes by Brugnano et al.\,\cite{Bru24}}
We also consider the recent code {\sc fhbvm2} available at
\url{https://people.dimai.unifi.it/brugnano/fhbvm/}.
This code allows only one single fractional derivative so that
we set $\alpha =0.8$ for both equations.
Setting $T=220$ and $N=300$ steps, we found an approximate solution in about $4.6$ seconds with 
a very moderate peak of memory allocation of about $1$ Megabyte.
With $T=1000$ and $N=1200$ the CPU time is about $35.5$ seconds.
With a higher $\alpha > 1$, the code becomes slower.

With our code we integrate the problem 
till $T=220$ in $0.8$ seconds and till $T=1000$ in $3.2$ seconds, with an accuracy of about $10^{-7}$. However, since our code is implemented in {\sc Fortran}, making a true comparison is not possible. 

\subsection{Example 3: a multi-term problem}
\label{sect:multi}

A benchmark problem from \cite{XueBai17} (see also \cite[Eq.\ (35)]{Gar18})
is given by ($0<\alpha <1$)
\begin{equation}\label{example3}
\dddot y(t) + D_*^{\alpha + 2}y(t) + \ddot y(t) + 4\, \dot y(t) +
D_*^{\alpha}y(t) + 4\, y(t) = 6 \,\cos{t},
\end{equation}
with initial conditions $y(0) = 1, \ \dot y(0) = 1, \ \ddot y(0) = -1$, so that
its exact solution is 
\begin{equation} \label{sol:ex3}
y(t) = \sqrt{2} \sin\left(t+\frac{\pi}{4}\right) 
\end{equation}
independent of $\alpha$. For studying the stability of this
linear differential equation we consider the characteristic equation of the homogeneous problem, which is given by
\begin{equation} \label{eq:chareqmt}
\mathcal{L}_\alpha(z) = z^{\alpha +2}+z^{\alpha }+z^3+z^2+4 z+4 = 0.    
\end{equation}
A numerical investigation reveals that at a value
$\alpha^* \in ( 0.654298, 0.654299 ) $
a pair of complex conjugate roots of \eqref{eq:chareqmt} crosses the imaginary axis from the left complex plane to the right one.
The two roots of $\mathcal{L}_\alpha (z)$ are approximately
$\pm 1.65686\, \iu$. For $\alpha > \alpha^*$ the solution of \eqref{example3}
is unstable, while for $\alpha \le \alpha^*$ it is stable.

For a numerical computation we insert
the definition of the Caputo derivative into \eqref{example3}
and we write the problem
as a first order system (three differential equations and one algebraic
equation) which gives
\[
\begin{array}{rcl}
\dot y_0 (t) &=& y_1(t) ,\quad \dot y_1 (t) ~=~ y_2(t), 
\quad \dot y_2 (t) ~=~ y_3(t)  \\[1.5mm]
0 &=& y_3(t) + I_{\alpha} (y_3)(t) + y_2(t) +4\, y_1(t) + 
I_{\alpha} (y_1)(t) + 4\, y_0(t) - 6\, \cos t ,
\end{array}
\]
where, with $\alpha_0 = 1 - \alpha$,
\[
I_\alpha (y) (t) ~=~ \frac 1 {\Gamma (1-\alpha )}\int_0^t (t-s)^{-\alpha} y(s)\, \d s
~=~ \frac 1 {\Gamma (\alpha_0 )}\int_0^t (t-s)^{\alpha_0 -1} y(s)\, \d s .
\]
This system is of the form \eqref{eq:Voltgeneral} and can be solved by
the approach of the preceding sections. As suggested in~\cite{Gar18}
we consider $\alpha = 1/2$ and $T=5000$.
We apply the code {\sc radau5} with the fast
linear algebra routines and $\eps = {\it Tol} = 10^{-5}$ to this system.
With $15\,812$ accepted steps the computation gives an error 
$0.11\cdot 10^{-5}$ at the final point, and it needs $0.4$ seconds cpu time.
For values $0<\alpha \le \alpha^*$ we get similar accurate results.
However, for
larger values,  such as $\alpha = 0.655 > \alpha^*$, we observe an error
$0.499 \cdot 10^3$ demonstrating the instability of the equation.

\section{Solving 1D fractional partial differential equations}
\label{sect:1DPDE}

We conclude by considering the following fractional PDE
with Dirichlet boundary conditions
\begin{equation} \label{1Dpde}
\begin{array}{rcl}
D_{\ast}^{\alpha} u(x,t) & = & \displaystyle
\frac{\partial^2 u}{\partial x^2} (x,t) + f(x,t) \\[3mm]
u(0,t) &=& u(1,t)  ~= ~0
\end{array}\qquad \mbox{for} \quad x \in (0,1),\quad t > 0 .
\end{equation}
Here, $D_*^\alpha$ denotes the fractional derivative with respect to
time $t$.
For $0<\alpha <1$ we impose an initial condition $u(x,0) = u_0(x)$,
similar as for parabolic problems, and for $1<\alpha <2$ we impose
$u(x,0) = u_0(x)$ and $\frac{\partial u}{\partial t }(x,0) = u_0^{(1)}(x)$
as for hyperbolic problems.

With the method of lines we transform this problem into a fractional
ordinary differential equation. We consider an
equi-spaced grid $x_i = i \,\Delta x$ for $i=1,\ldots,d$ and
$\Delta x = {1}/(d+1)$, and we apply central finite differences
to the space derivative. For the solution approximation
$y_i(t) \approx u (x_i,t)$ at the grid points this yields the system
\begin{equation}\label{eqmol}
D_*^\alpha y_i(t) = \frac 1{\Delta x^2} \Bigl( y_{i+1}(t) - 2 \,y_i(t)
+ y_{i-1} (t) \Bigr) + f(x_i,t),\qquad i=1,\ldots , d ,
\end{equation}
with $y_0(t) = y_{d+1}(t) = 0$, and initial values
$y_i(0) = u_0(x_i)$ for $0<\alpha <1$ and 
$y_i(0) = u_0(x_i)$, $\dot y_i(0)= u_0^{(1)} (x_i)$ for
the case $1<\alpha <2$. According to \eqref{Volt2}, the fractional
differential equation \eqref{eqmol} becomes, for $0<\alpha < 1$,
\begin{equation}\label{voltpde}
y_i(t) = u_0(x_i) + I_i (t), \qquad I_i (t) = \frac 1{\gamma (\alpha )}
\int_0^t (t-s)^{\alpha -1} G_i\bigl( y(s)\bigr) \, \d s ,
\end{equation}
where $G_i\bigl( y(s)\bigr)$ represents the right-hand side of \eqref{eqmol}.
The dimension $d$ of this system~is typically very large and even the
fast linear algebra explained in Section~\ref{sect:sumexp} is not efficient
enough. We have to exploit the special structure of the
arising Jacobian matrix.

\subsubsection*{Exploiting banded structures}

We say that the problem \eqref{eq:Voltgeneral} has {\it banded structure},
if the mass matrix $M$ and $\partial F / \partial I$ are diagonal matrices,
and the matrices $\partial F / \partial y$ and $\partial G /\partial y$
are both banded with upper and lower bandwidths $b_u$ and $b_l$,
respectively.

Note that the system \eqref{voltpde} has a tridiagonal banded structure.
In fact, the matrix $M$ is the zero-matrix, $\partial F/\partial I$ is the
identity, the matrix $\partial F / \partial y$ is also a zero matrix, and
the matrix $\partial G / \partial y$ is tridiagonal.

\smallskip

\begin{lemma}
Consider a problem that has banded structure with band widths $b_u$ and
$b_l$. Then, the matrix $\hat J$ of \eqref{eq:jhat} is banded
with the same band widths.
\end{lemma}

\begin{proof}
By assumption the matrix $J$ is banded. Moreover, the only non-zero
element of ${\partial F}/{\partial I_j}$ is at position $j$, and
${\partial G_j}/{\partial y}$ is the $j$th row of a banded matrix.
Therefore, every rank-one matrix in the sum of \eqref{eq:jhat}
is banded with the same bandwidths as $J$. This proves the statement.
\end{proof}
\smallskip

This lemma shows that we can exploit the banded structure in our
fast linear algebra routines.

\subsubsection*{Use of fast linear algebra for banded problems}

In the calling program one has to define the dimension {\sc n}
as in Section~\ref{sect:sumexp} (or as in Section~\ref{sect:sumexp2}
if some of the exponents satisfy $\alpha_j > 1$). For banded
structures one has to put {\sc ijac}$=1$, {\sc mljac}$=b_l$, and
{\sc mujac}$=b_u$ (the code requires $b_u \ge 1$).

The information of the Jacobian has to be stored in
{\sc dfy(ldfy,n)} as follows. The $b_u+1+b_l$ diagonals of
$J=\partial F/\partial y$ are stored as rows in such a way that columns
of~$J$ remain columns in {\sc dfy}. In this way the element {\sc dfy(1,1)}
is not used by $J$. We put {\sc dfy(1,1)}$=d$, so that the number of
columns is provided. As in Section~\ref{sect:sumexp} there are $d_I$
blocks corresponding to the $d_I$ integral terms in the problem.
New is the storage in the first two columns of each block, the rest is
not changed. The first column of the $j$th block contains the
$j$th element of $\partial F/\partial I_j$ in position $b_u+1$.
The second column contains the relevant information of $\partial G_j/\partial y$
in the positions $1$ to $b_u+1+b_l$. We finally put
\,{\sc dfy(3,n)}$=-d_I$\, to provide also the number of integrals.

\subsubsection*{Numerical experiment}

For our numerical illustration we choose
\begin{equation} \label{eq:f3}
f(x,t) = \frac12 x \left( 1 - x \right) 
\frac{ \Gamma (\beta)\, \beta  }
{\Gamma \left(\beta + 1 - \alpha\right)}\, t^{\beta - \alpha}
+ \left( t^\beta + 1 \right) ,\qquad  \beta \ge \alpha ,
\end{equation}
for which the
exact solution of \eqref{1Dpde} is known to be
\begin{equation} \label{eq:solpde}
u(x,t) = \frac12 x \left( 1 - x \right) \left( t^\beta + 1 \right)  .
\end{equation}
We choose $\alpha = 1/3$, $\beta = 5/3$, and we consider the
numerical integration over the interval $[0,1000]$.
With $\eps = 10^{-6}$
the fractional kernel $t^{\alpha -1}$ is approximated by a sum of
exponentials with parameters $M=-49$ and $N=77$ obtained
from Algorithm~\ref{alg_para}.

\begin{table}[ht!] 
\begin{center}
{\small
\begin{tabular}{|c|c|c|c|c|c|}
\hline
 \rule[0mm]{0cm}{3.0mm}%
 $d = d_I$  & $100$ & 
$300$ & $1000$ & $3000$ & $10000$\\
\hline
\rule[0mm]{0cm}{3.6mm}%
cpu (banded) &  $0.65\cdot 10^{-1}$ &  $0.20\cdot 10^{0}$  
& $0.68\cdot 10^{0}$ & $0.20\cdot 10^{1}$ & $0.69\cdot 10^{1}$  \\
${\rm err}$ &  $0.11\cdot 10^{-7}$ &   $0.19 \cdot 10^{-7}$  
& $0.46 \cdot 10^{-8}$ & $0.64 \cdot 10^{-7}$  & $0.11 \cdot 10^{-6}$  \\
  \hline
\end{tabular}
\vspace{1.5mm}
\caption{Numerical results for the 1D partial differential equation
using the banded linear algebra routines}
\label{tab:pde} 
}
\end{center}
\vspace{-3mm}
\end{table} 

The cpu time and the relative error obtained
by the code {\sc radau5} with ${\it Tol} = \eps = 10^{-6}$ and
with the banded fast linear algebra option
are given in Table~\ref{tab:pde}. We can observe that the cpu
time increases linearly with the dimension $d$ of the system.
Independent of the dimension, the code takes about $43$
steps with
step sizes increasing from $\Delta t\approx 0.04$
in the beginning to $\Delta t \approx 137$ at the end of the interval.

\subsubsection*{Further fractional PDEs}

Systems of reaction diffusion equations with memory can be modeled
by fractional equations of the from (see e.g., \cite{LF08})
\begin{equation} \label{fpdechem}
\begin{array}{rcl}
D_{\ast}^{\alpha} u_1(x,t) & = & \displaystyle
K\frac{\partial^2 u_1}{\partial x^2} (x,t) - k_1 u_1(x,t) u_2(x,t)
+ (k_2+k_3) u_3(x,t) \\[3mm]
D_{\ast}^{\alpha} u_2(x,t) & = & \displaystyle
K\frac{\partial^2 u_2}{\partial x^2} (x,t) - k_1 u_1(x,t) u_2(x,t)
+ k_2 u_3(x,t) \\[3mm]
D_{\ast}^{\alpha} u_3(x,t) & = & \displaystyle
K\frac{\partial^2 u_3}{\partial x^2} (x,t) + k_1 u_1(x,t) u_2(x,t)
- (k_2+k_3) u_3(x,t) 
\end{array}
\end{equation}
for $x\in (0,1)$ and $t> 0$. We consider homogeneous Dirichlet boundary
conditions, suitable initial functions, and parameters $K=0.5$,
$k_1=1$, $k_2=2$, and $k_3 = 3$. After a semi-discretization as in
\eqref{eqmol} we obtain a system of fractional ODEs for
$y=\bigl( y_1^1,\ldots ,y_d^1,y_1^2,\ldots ,y_d^2,y_1^3,\ldots ,y_d^3\bigr)$,
where $y_i^j(t)\approx u_j(x_i,t)$.
Due to the reaction
term the Jacobian matrix is no longer banded. However, since the
reaction term is not stiff (when compared to $\partial^2/\partial x^2$),
we can remove all elements that are not on the diagonal or on the first
super- and subdiagonals. The neglection of such elements increases slightly
the number of simplified Newton iterations, but the decrease in cpu
time is remarkable. Similar as in Table~\ref{tab:pde} we observe
that the cpu time scales linearly with the number of grid points.

In the case where the reaction is stiff, such a reduction of the Jacobian
matrix to tridiagonal form is not recommended. If one orders the
elements of the vector $y$ according to
$y=\bigl( y_1^1,y_1^2,y_1^3,\ldots , y_d^1,y_d^2,y_d^3\bigr)$, then
the Jacobian matrix becomes banded with bandwidths $b_l = b_u = 3$.
This allows for an efficient treatment of the linear algebra without
neglecting elements of the Jacobian matrix.

\begin{table}[ht!] 
\begin{center}
{\small
\begin{tabular}{|c|c|c|c|c|c|}
\hline
 \rule[0mm]{0mm}{3.5mm}%
 {\it }  & {\it nfcn} & 
{\it njac} & {\it nstep} & {\it error} & {\it cpu time}\\
\hline
 \rule[0mm]{0cm}{3.6mm}%
{\it 3-diagonal} &  $274$ &   $29$ & $29$ & $8.21\cdot 10^{-6}$ & 1.63 \\
 \rule[0mm]{0cm}{3.0mm}%
{\it 7-diagonal} &  $183$ &   $4$ & $28$ & $9.92\cdot 10^{-6}$ & 1.17 \\
  \hline
\end{tabular}
\vspace{1.0mm}
\caption{Numerical comparison of different approximations of the Jacobian
matrix}.
\label{tab:37} 
}
\end{center}
\vspace{-6mm}
\end{table} 

For our numerical experiment we consider initial values
$u_1(x,0)=0.5\, x (1-x)$, $u_2(x,0)=x^2(1-x)$,
$u_3(x,0)=1.5\, x(1-x)^2$, and an equidistant grid with
$d=1000$ interior grid points. We fix ${\it Tol} = \eps= 10^{-5}$,
and we compute the parameters of the kernel approximation with
Algorithm~\ref{alg_para}. For $\alpha = 1/2$ and $t\in [0,30]$ this gives
$M=-39$ and $N=37$, so that the complete system \eqref{eq:ODE},
which consists of $3d$ $y$-variables and of $3d$ integral terms,
is of dimension $3d + 3d(N-M) = 231 d = 231\,000$.
Table~\ref{tab:37} presents the statistics
of the code {\sc Radau5} for the two approaches:
neglecting some elements of the reaction term ($3$-diagonal Jacobian
matrix) and considering the complete Jacobian with an ordering of
variables to obtain a $7$-diagonal structure. We see that
the number of steps {\it nstep} and the relative {\it error} are
more or less identical. The number of function evaluations {\it nfcn}
and the number of Jacobian evaluations {\it njac} are larger
for the $3$-diagonal version. This is explained by the fact that
due to the non-exact Jacobian the code requires about~$3$
simplified Newton iteration per step in contrast to the
$7$-diagonal version, which needs only $2$ simplified Newton iterations
per step.
This is also reflected in the {\it cpu time}, which is given in seconds.

\section{Conclusions}

In this article we have described an efficient memoryless me\-tho\-do\-logy to deal with fractional 
differential equations possibly coupled with  stiff ordinary and differential-algebraic
equations. 

The methodology is based on a suitable expansion of the fractional kernels in terms
of exponential functions, which allows to transform the problem into an augmented set of stiff ODEs of large dimension.
Our numerical experiments confirm the efficiency of this methodology which
allows to get a much more versatile code, able to deal with general problems,
a feature that
makes it very appealing in disciplines where fractional differential equations have to
be solved accurately and efficiently.

\subsection*{Acknowledgments}

Nicola Guglielmi acknowledges that his research was supported by funds from the Italian 
MUR (Ministero dell'Universit\`a e della Ricerca) within the 
PRIN 2022 Project ``Advanced numerical methods for time dependent parametric partial differential equations with applications'' and the PRIN-PNRR Project ``FIN4GEO''.
He is also affiliated to the INdAM-GNCS (Gruppo Nazionale di Calcolo Scientifico).

Ernst Hairer acknowledges the support of the Swiss National Science Foundation, grant No.200020 192129.

\bibliographystyle{plain}

\end{document}